\documentclass[12pt]{article}
\usepackage{amsmath,amssymb}
\usepackage{amsthm}
\usepackage{graphicx,tikz}
\usepackage{enumerate}
\usepackage{titlesec}
\usepackage{caption}
\usepackage{xcolor}
\usepackage[sans]{dsfont}

\numberwithin{equation}{section}

\newtheorem{Thm}{Theorem}[section]
\newtheorem*{Thm*}{Theorem}
\newtheorem{Prop}[Thm]{Proposition}
\newtheorem*{Prop*}{Proposition}
\newtheorem{Lem}[Thm]{Lemma}

\newtheorem{Cor}[Thm]{Corollary}
\newtheorem{Fact}[Thm]{Fact}

\newtheorem{Exa}[Thm]{Example}

\theoremstyle{remark}

\theoremstyle{definition}

\newtheorem{Def}[Thm]{Definition}

\newtheorem*{Def*}{Definition}

\numberwithin{equation}{section}

\newcommand{\m}[1]{\mathbb{ #1}}
\newcommand{\mc}[1]{\mathcal{ #1}}

     \def\ol{\overline}    
   
\def\al{\alpha}       \def\be{\beta}        \def\ga{\gamma}
\def\de{\delta}       \def\eps{\varepsilon}  \def\ze{\zeta}
       \def\la{\lambda}      
\def\si{\sigma}                
\def\ph{\varphi}

\theoremstyle{definition}

\theoremstyle{remark}

\newtheorem{Rmq}[Thm]{Remark}

\numberwithin{equation}{section}
\newfont{\goth}{eufm10 at 12pt}
\newfont{\gots}{eufm8 at 9pt}

\def\bt{\begin{Thm}}
\def\et{\end{Thm}}
\def\br{\begin{Rmq}}
\def\er{\end{Rmq}}

\def\bc{\begin{Cor}}
\def\ec{\end{Cor}}
\def\bp{\begin{Prop}}
\def\ep{\end{Prop}}
\def\bl{\begin{Lem}}
\def\el{\end{Lem}}
\def\bd{\begin{Def}}
\def\ed{\end{Def}}
\def\bq{\begin{quotation}}
\def\eq{\end{quotation}}
\def\bfa{\begin{Fact}}
\def\efa{\end{Fact}}
\def\bexa{\begin{Exa}}
\def\eexa{\end{Exa}}

\def\ra{\rightarrow}
\def\vs{\vspace{1em}}

\begin{document}
\title{
Convolution and square 
in abelian groups III
}
\author{Yves Benoist
}
\date{}

\maketitle

\begin{abstract}
\noindent 
We know that 
the  functional equa\-tion $f\!\star\! f (2\,t) =\la f(t)^2$ on the cyclic group of odd order $d$  has a non-zero solution $f$ when $\la =\sqrt{a}\!+\!i\sqrt{b}$ where $a$, $b$ are positive integers with $a\!+\!b=d$ and $a\equiv \frac{(d+1)^2}{4}\; {\rm mod}\;4$. We show here that in this case 
the function $f$ can be chosen to be equal to the conjugate of its Fourier transform.  
\end{abstract}

\renewcommand{\thefootnote}{\fnsymbol{footnote}} 
\footnotetext{\emph{2020 Math. subject class.}  Primary 11F03~; Secondary  11F27} 
\footnotetext{\emph{Key words} Finite Fourier transform, Cyclic group, Elliptic curve, Complex multiplication, Theta function, Jacobi sums, Weil numbers.}     
\renewcommand{\thefootnote}{\arabic{footnote}} 
\newpage

{\footnotesize \tableofcontents}\nopagebreak
\newpage

\section{Introduction}
\label{secintexp}

\subsection{Critical values}
\label{seccrival}

This paper is the sequel of \cite{CSAGI} in which we introduced 
the notion of critical values. 
Let us first recall this notion.
Let $G$ be a finite abelian group of odd order $d$. 
Denote by $\mc F_G$ the space of complex valued functions on $G$.
For  $\la\in \m C$, we denote by $\mc F_\la=\mc F_{G,\la}$  the set:
\begin{equation}
\label{eqntante2} \textstyle
\mc F_\la\; :=\; \{ f: G\ra \m C\mid
f\!*\!f(2k)
\: =\; \la \, f(k)^2
 \;\;\; \mbox{\rm for all $k$ in $G$}\}
\end{equation}
A value $\la$ for which such a non-zero function $f$ 
exists is called a ``critical value on $G$'', 
or a ``$d$-critical value''
when $G$ is the cyclic group $\m Z/d\m Z$. The non-zero elements of $\mc F_\la$
are called $\la$-critical functions.
We denote by $\mc C_G$ the set of critical values on $G$ and simply by $\mc C_d$
the set of $d$-critical values. 
According to \cite[Prop. 2.1]{CSAGI}, the set $\mc C_G$ is finite and all 
the elements $\la$ of $\mc C_G$ are odd algebraic integers, i.e. the 
numbers $\frac{\la-1}{2}$ are algebraic integers. 
\vs 

The main result \cite[Thm 2.3]{CSAGI} of this first paper  is the following.

\begin{Prop}
\label{prorapirb}
Let  $G\!=\!\m Z/d\m Z$ with $d$ odd, $a$, $b$ be positive integers with 
$a\!+\!b\!=\!d$ and  
$a\!\equiv\!\frac{(d+1)^2}{4}\;{\rm  mod}\; 4$.
The values $\la_0:=\sqrt{a}\pm i\sqrt{b}$
are  $d$-critical.
\end{Prop}


\subsection{Fourier transform}
\label{secfoutra}

In the present paper we still focus on the cyclic group $G=\m Z/d\m Z$ and we want to study the action of the Fourier transform on the $\la$-critical functions.
We recall that the Fourier transform $f\mapsto \widehat{f}$ on
$\mc F_G$ is given by,
 for all  $k\in G$,
\begin{equation}
\label{eqnfoufor}
\textstyle
\widehat{f}(k)= \frac{1}{\sqrt{d}}\sum\limits_{\ell\in G}e^{-2i\pi k\ell/d}\;f(\ell).
\end{equation}
See for instance \cite{TerrasFourier}. The Fourier transform $f\mapsto \widehat{f}$ sends $\mc F_\la$ onto $\mc F_{d/\la}$.
One also defines the conjugate Fourier transform $f\mapsto \overline{\widehat{f}}$ on
$\mc F_G$ by
\begin{equation}
\label{eqnfoubarfor}
\textstyle
\overline{\widehat{f}}(k)= \frac{1}{\sqrt{d}}\sum\limits_{\ell\in G}e^{2i\pi k\ell/d}\;\overline{f(\ell)}.
\end{equation}
This map $f\mapsto \overline{\widehat{f}}$ is an antilinear involution of $\mc F_G$,
i.e. one has the equivalence 
\begin{equation}
\label{eqnfoubar0}
g=\overline{\widehat{f}}\Longleftrightarrow 
f=\overline{\widehat{g}}, 
\;\;\mbox{for all $f$, $g$ in $\mc F_G$.}
\end{equation}
When studying experimentally in \cite{CSAGI} the $d$-critical values, it seemed that the most interesting ones are those which have absolute value 
$|\lambda|=\sqrt{d}$.
Moreover, for all $\la$  with $|\la|=\sqrt{d}$, this map  $f\mapsto \overline{\widehat{f}}$ preserves the variety $\mc F_\la$ of $\la$-critical functions. 
Therefore it is natural to introduce the subsets
\begin{equation} 
\label{eqnbenoisto} 
\mc B^o_d:=\{ \mbox{$\la\in \mc C_d \mid$
there exists  $f\neq 0$ in $\mc F_\la$  
such that $\overline{\widehat{f}}=f$}\},\;\;{\rm and}
\end{equation} 
\begin{equation} 
	\label{eqnbenoist} 
\hspace*{-1.5ex}\mc B_d:=\{ \mbox{$\la\in \mc C_d \mid$
		there exist
		$q\!\in\!(\m Z/d\m Z)^*$, $f\neq 0$ in $\mc F_\la$  
		with $\overline{\widehat{f}}=f_q$ }\},
\end{equation} 
where $f_q$ is the function $f_q:k\mapsto f(qk)$
\vs 

We begin by a few preliminary remarks on these two finite sets~:

\br
It is equivalent in \eqref{eqnbenoist} to require that there exists a $\la$-critical function $h$ 
on $G$ such that $\overline{\widehat{h}}=\al h_q$
for some scalar $\al\in \m C$. Indeed, by
\eqref{eqnfoubar0}, this scalar $\al$ must have absolute value $|\al|=1$. Therefore writing $\al=\overline{\be}/\be$ with $\be$ in $\m C$,
the function $f:=\be h$ will satisfy $\overline{\widehat{f}}=f_q$.
\er

\br 
According to the previous discussion, all the elements $\la$ of 
$\mc B_d$ have absolute value $|\la|=\sqrt{d}$. 
\er

\br 
When $\la$ is in $\mc B_d$ then its complex conjugate $\ol{\la}$ is also in $\mc B_d$.\\
When $\la$ is in $\mc B_d^o$ with a symmetric function $f$,
then $\ol{\la}$ is also in $\mc B_d^o$.
\er

\br 
\label{remgausum}
The simplest element of $\mc B_d^o$ is $\la=\sqrt{d}$ for $d\equiv 1$ mod $4$
and is $\la=\pm i\sqrt{d}$ for $d\equiv 3$ mod $4$.
The corresponding $\la$-critical function is the Gaussian function $h:=k\mapsto \eta^{\pm k ^2}$ where $\eta=-e^{i\pi d }$. This will be explained with more details in Lemma \ref{lemgaufun}.
\er 

\br 
\label{remdircha}
The second simplest elements of $\mc B_d^o$ are $\la=\chi(4)\,J(\chi,\chi)$ where $\chi$ is a Dirichlet character modulo $d$ whose square $\chi^2$ is primitive and $J(\chi,\chi)$ its Jacobi sum.
The corresponding $\la$-critical function is the Dirichlet character $\chi$. This will be explained with more details in Lemma \ref{lemdircha}.
\er 

\br
We will see in Section \ref{secgaufun} that when $d=1$ mod $4$
the $d$-critical value $\la_0=-\sqrt{d}$ is in $\mc B_d$. For $d=5$,
one can check that
this value is not in $\mc B_d^o$. This example explains why it is natural to introduce both $\mc B^o_d$ and $\mc B_d$.
\er

The main result of this short paper is to prove that the $d$-critical values introduced in \cite{CSAGI} belong to  $\mc B_d^o$.

\bt
\label{thmfoubar}
Let  $G=\m Z/d\m Z$ with $d$ odd, $a$, $b$ be positive integers with 
$a\!+\!b\!=\!d$ and  
$a\!\equiv\!\frac{(d+1)^2}{4}\;{\rm  mod}\; 4$.
Then both values $\sqrt{a}\pm i\sqrt{b}$ belong to $\mc B^o_d$
\et

The proof of Theorem \ref{thmfoubar}
does not rely on \cite{CSAGI} nor on \cite{CSAGII}. In particular,
it gives a new proof of Proposition \ref{prorapirb}.
As in \cite{CSAGI}, the proof of this elementary statement relies 
on elliptic curves with complex multiplication and on the modularity properties of their theta functions. 
The main surprise is the fact that the function $f_{z}$ introduced in Proposition \ref{profoubar2} 
satisfies Equation \eqref{eqnfoubar2}.
And also
that the analogous functions $f_{0,z}$ and $f_{1,z}$ 
introduced in Proposition \ref{profoubar3} satisfy Equations \eqref{eqnfoubar30} 
and \eqref{eqnfoubar31}.

Other surprising aspects of the Fourier transform on cyclic groups can be found in 
\cite{BjorckSaffari}, \cite{TaoCyclic}, \cite{HaagerupFourier} and
\cite{FuhrRzeszotnik}.

\subsection{Examples}

Relying on the numerical experiments of 
Section \ref{secliscri},
one can describe the sets $\mc B_d^o$ and $\mc B_d$ for small values of $d$.
Here is the result for $d\leq 13$.

\vs 

\noindent 
$\star$ \fbox{$d=3$} \hspace{1em} $\mc B_3^o=\{\pm i\sqrt{3}\}$
and $\mc B_3=\mc B_3^o$.

\noindent 
$\star$  \fbox{$d=5$}  \hspace{1em}  $\mc B_5^o=\{\sqrt{5}
\; ,\; 1\!\pm\! 2i\}$ and $\mc B_5=\mc B_5^o\cup\{-\sqrt{5}\}$. 

\noindent 
$\star$  \fbox{$d=7$} \hspace{1em} 
$\mc B_7^o=\{ \pm i \sqrt{7} \; ,\; 
\pm 2\!\pm\! i\sqrt{3}\}$ and $\mc B_7=\mc B_7^o$.

\noindent 
$\star$ \fbox{$d=9$} \hspace{1em} 
$\mc B_9^o=
\{ 3\; ,\; \pm \sqrt{5}\!\pm\! 2i\; ,\;\pm 1\!\pm\! 2i\sqrt{2}\}$
and $\mc B_9=\mc B_9^o$.

\noindent $\star$ \fbox{$d=11$} \;  $\mc B_{11}^o=
\{\pm i\sqrt{11} ,\; 
2\!\pm\! i\sqrt{7},\;\pm 2\sqrt{2}\!\pm\! i\sqrt{3},$\\
\hspace*{8em}
$\pm (1\!+\!\eps\sqrt{5})\pm i\sqrt{5\!-\!2\eps \sqrt{5}}\;
\;{\rm for}\; \eps=\pm 1\}$ and $\mc B_{11}=\mc B_{11}^o$.

\noindent $\star$ \fbox{$d=13$} \;   $\mc B_{13}^o\supset
\{\sqrt{13},\;\pm 3\!\pm\!2i,\; \sqrt{5}\!\pm\! 2i\sqrt{2},\;\pm 1\!\pm\! 2i\sqrt{3}\}$ and\\ 
\hspace*{5em}
$\mc B_{13}\supset\mc B_{13}^o\cup\{-\sqrt{13}, -\sqrt{5}\pm 2i\sqrt{2}\}$.
\vs 

Here are a few comments on this list~:\\
- One notices that the values $-1\!\pm\! 2i$ ,  $-3$  and  $-2\!\pm\! i\sqrt{7}$ are not in the corresponding $\mc B_{d}$ because they are not $d$-critical values.\\ 
- Many of the critical values in this list are in $\mc B_d^o$ because of 
either Remark \ref{remgausum}, Lemma \ref{lemdircha}, or
Theorem \ref{thmfoubar}. \\
- For the $d$-critical values
$\la=-\sqrt{5}\!\pm\! 2i$, $-1\!\pm\! 2i\sqrt{2}$, 
$-2\sqrt{2}\!\pm\! i\sqrt{3}$,  
one can  construct explicitely symmetric $\la$-critical functions $f$
which satisfy $\overline{\widehat{f}}=f$.\\
- For the $d$-critical values 
$\la=-\sqrt{5}\!\pm\! 2i\sqrt{2}$, one can construct explicitely
symmetric $\la$-critical functions that satisfy 
$\overline{\widehat{f}}(k)=f(2k)$ for all $k$ in $\m Z/13\m Z$. 
But there exist no symmetric $\la$-critical functions
that satisfy  $\overline{\widehat{f}}=f$.\\
- For $d=13$ the above inclusions are probably equalities.\vs

As explained in \cite[\S 1.2]{CSAGI}, one might look at Equation \eqref {eqntante2} on  locally compact abelian groups $G$. When the group is $G=\m R$
or $G=\m Z$, the only $L^2$-solutions that I know are gaussian functions $f(t)=e^{at^2+bt+c}$ where $a,b,c \in \m C$,  $Re(a)<0$, 
together with, when $G=\m Z$, their restrictions to subgroups. 
When the group is $G=\m R/\m Z$, the only $L^2$-solutions that I know are proportional to a characters $\chi_n:t\mapsto e^{2i\pi nt}$ with $n\in \m Z$.

\subsection{Organization}

In Section \ref{secthecri} we prove Theorem \ref{thmfoubar}.\\
In Section \ref{secthecha} we give a new proof of Proposition \ref{prorapirb}.\\
In Section \ref{secproobd}, we explain in more detail  the above preliminary remarks. 

\section{Theta functions as critical functions}
\label{secthecri}
\bq 
In this Chapter we prove Theorem \ref{thmfoubar} 
and the more precise
Corollary \ref{corfoubar2} taking for granted 
Proposition \ref{profzcri}.
\eq 

\subsection{Theta functions} 

We recall the definition of the Jacobi theta function:
$$\textstyle
\theta_\tau (z)=\theta(z,\tau):=\sum\limits_{m\in \m Z}e^{i\pi\tau m^2}e^{2i\pi mz},
\;\;\mbox{\rm for $z\in \m C$ and $\tau\in \m H$,}
$$
where $\m H$ is the upper half plane 
$\m H=\{ \tau\in \m C\mid Im(\tau)>0\}$. 
This function is $1$-periodic: $\theta_\tau(z+1)=\theta_\tau(z)$.
We can now recall  the construction of $\la$-critical functions
on $\m Z/d\m Z$  from \cite[Prop. \!3.2]{CSAGI}.

\begin{Prop}
\label{profzcri}
Let $G\!=\!\m Z/d\m Z$ with $d$ odd,  $a$,$b$ be positive integers with 
$a\!+\!b\!=\!d$ and $a\!\equiv\!\frac{(d+1)^2}{4}\;{\rm  mod}\; 4$.
Let $\la_0:=\sqrt{a}\!+\!i\sqrt{b}$,
$
\displaystyle 
\tau_0:=\tfrac{\la_0^2-d^2}{4d^2}
$
and $z\in \m C$. Then the  function $f_{z}:\ell\mapsto \theta(\la_0z\!+\!\frac{d+1}{2d}\ell,\tau_0)$ is $\la_0$-critical
on $G$.
\end{Prop}

A new proof of Proposition \ref{profzcri} will be given in Chapter \ref{secthecha}.
Note that we have omitted 
the dependence on $\tau_0$ in the notation $f_z$.\vs 

\br 
In \cite[Prop. \!3.2]{CSAGI} one deals with the function $\ell\mapsto \theta(z\!+\!\tfrac{\ell}{d},\tau_0)$
where the factor $\la_0$ and the factor $\tfrac{d+1}{2}$ do not occur. Note that, since $d$ is odd, 
the factor $\tfrac{d+1}{2}$ is nothing than the inverse of $2$ in the group $\m Z/d\m Z$. Introducing these factors does not change the conclusion of Proposition
\ref{profzcri} but will be crucial for the validity of 
Proposition \ref{profoubar2}.
\er

\subsection{Theta functions and Fourier transform}

In Proposition
\ref{profoubar2} the real numbers $a$ and $b$ are not assumed to be integers.

\begin{Prop}
\label{profoubar2}	Let $G\!=\!\m Z/d\m Z$ with $d$ odd,  $a$,$b$ be positive real numbers with 
$a\!+\!b\!=\!d$.
Let $\la_0:=\sqrt{a}\!+\!i\sqrt{b}$,
$\displaystyle \tau_0:=\tfrac{\la_0^2-d^2}{4d^2}$
and $z\in \m C$.
Then the function $f_{z}:\ell\mapsto \theta(\la_0z\!+\!\frac{d+1}{2d}\ell,\tau_0)$  on $G$ satisfies
\begin{equation}
\label{eqnfoubar2}
\textstyle
\overline{\widehat{f}}_{\overline{z}}
\;=\;
\frac{\la_0}{\sqrt{d}} \; e^{4i\pi d^2z^2}\, f_{z}.
\end{equation}
\end{Prop}

Note that the function on the left-hand side is indexed by the complex conjugate $\overline{z}$. Note also that, when $a$, $b$ are not assumed to be integers, even when $z=0$, this function $f_{z}$ is not 
always $\la_0$-critical.

Theorem \ref{thmfoubar} follows from the following corollary of Proposition \ref{profoubar2} which tells us that, when $z=r$ is real, with the precise factors $\la_0$ and $\tfrac{d+1}{2}$ as above, 
the $\la_0$-critical functions $f_r$ are proportional 
to their image $\overline{\widehat{f}}_r$. 

\begin{Cor}
\label{corfoubar2}
Let $G\!=\!\m Z/d\m Z$ with $d$ odd,  $a$,$b$ be positive integers with 
$a\!+\!b\!=\!d$ and $a\!\equiv\!\frac{(d+1)^2}{4}\;{\rm  mod}\; 4$.
Let $\la_0:=\sqrt{a}\!+\!i\sqrt{b}$,
$\displaystyle \tau_0:=\tfrac{\la_0^2-d^2}{4d^2}.$
When $r$ is real, the $\la_0$-critical function $f_{r}:\ell\mapsto \theta(\la_0r\!+\!\frac{d+1}{2d}\ell,\tau_0)$  
on $G$ satisfies
\begin{equation}\label{eqnfoubar1}\textstyle
\overline{\widehat{f}}_{r}=\frac{\la_0}{\sqrt{d}} \,e^{4i\pi d^2r^2}\, f_{r}.
\end{equation}
\end{Cor}

\subsection{Preliminary formulas}

To prove Proposition \ref{profoubar2}
we need the following two classical formulas.

The first formula computes the Fourier transform of the theta function.

\bl 
\label{lemfouthe} For $w\in \m C$, $\tau\in \m H$, $d$ positive integer
and $k\in \m Z/d\m Z$, one has
\begin{equation}
\label{eqnfouthe}\textstyle
\sum\limits_{\ell \in \m Z/d\m Z}
e^{-2i\pi k \ell/d}\;\theta(w\!+\!\ell /d,\tau)\; =\;
d\,e^{i\pi k^2\tau}\,e^{2i\pi k w}\;\theta(dw+dk\tau,d^2\tau)
\end{equation}
\el

\begin{proof}Just write the left-hand sides as a double sum over $m$ in $\m Z$
and $\ell $ in $\m Z/d\m Z$ and notice that 
$\sum_{\ell \in \m Z/d\m Z}e^{2i\pi \ell (m-k)/d}$ is equal to $d$ when $d$ divides $m\!-\!k$ 
and is equal to $0$ otherwise to obtain the right-hand side. Here are the details
\begin{eqnarray*} 
LHS&=&\textstyle
\sum\limits_{\ell \in \m Z/d\m Z, m\in \m Z}
e^{i\pi\tau m^2}e^{2i\pi mw}e^{2i\pi \ell (m-k)/d}\\
&=& \textstyle d \sum\limits_{n\in \m Z}
e^{i\pi\tau (k+dn)^2}e^{2i\pi (k+dn)w}
\; =\; RHS
\end{eqnarray*}
that prove \eqref{eqnfouthe}.
\end{proof}

The second formula is the transformation formula due to Hecke that deals with an element 
$\si =\mbox{\scriptsize $\left(\!
\begin{array}{cc} \al&\be   \\ \ga&\de\end{array}
\!\right)$} \in {\rm SL}(2,\m Z).
$

\bl 
\label{lemtrathe} If $\si\equiv  \mathds{1}$ {\rm mod} $2$,
and $\ga>0$, then, for all $w$ in $\m C$ and $\tau$ in $\m H$,
\begin{eqnarray}
\label{eqntrathe}
\theta(\tfrac{w}{\ga\tau+\de},\si\tau)
&=& i^\frac{\de-1}{2}(\!\tfrac{\ga}{\de}\!)\;
(\ga\tau\!+\!\de)^{\frac12}\;e^{i\pi \frac{\ga w^2}{\ga\tau+\de}}\;\theta(w,\tau).
\end{eqnarray}
\el

In this formula, the ${\rm SL}(2,\m Z)$ action 
on the upper half plane $\m H$ is the standard action 
$\si\tau =\frac{\al\tau+\be}{\ga\tau+\de}$,
the number $z^{\frac12}$
is the square root with positive real part of the complex number $z\in \m H$, and
the symbol  $(\!\frac{\ga}{\de}\!)=\pm 1$ is 
the Jacobi symbol.

\begin{proof} See for instance \cite[p.32]{Mum83}, \cite[p.148]{Kob84}   or \cite[Chap. 5]{Poli03}.
\end{proof}

\subsection{Proof of Proposition \ref{profoubar2} }

We begin by a Corollary of Lemma \ref{lemtrathe}.

\begin{Cor}
\label{cortrathe}	Let  $d$ be an odd integer,  $a$,$b$ be positive real numbers with 
$a\!+\!b\!=\!d$.
Let $\la_0:=\sqrt{a}\!+\!i\sqrt{b}$, and
$\displaystyle \tau_0:=\tfrac{\la_0^2-d^2}{4d^2}$.
Then, for all $z$ in $\m C$, one has 
\begin{equation}
\label{eqntrathe2}
\textstyle
\theta(d\overline{\la}_0z,-d^2\overline{\tau}_0)
\;=\; \frac{\la_0}{d}\, e^{4i\pi d^2z^2}\;
\theta(\la_0z,\tau_0)
\end{equation}
\end{Cor}
 
\begin{proof}[Proof of Corollary \ref{cortrathe}]
We apply Lemma \ref{lemtrathe} with the matrix 
$$
\si_0
=\mbox{\scriptsize $\left(\!
\begin{array}{cc} d^2&\frac{d^2-1}{4} \\4&1\end{array}\!
\right)$} \equiv {\bf 1}\; {\rm mod}\; 2,
$$ 
with $\tau=\tau_0$ and with $w=\la_0 z$. 
This matrix $\si_0$ has been chosen so that one has the equality 
\begin{equation} 
\label{eqnsitauo}
\si_0\tau_0 =- d^2\overline{\tau}_0. 
\end{equation}
We  note that $\la_0\overline{\la}_0=d$ and that 
$1/(4\tau_0\!+\!1)=\overline{\la}_0^2$ so that 
Formula \eqref{eqntrathe2} directly follows from 
Formula \eqref{eqntrathe}.
\end{proof}

\begin{proof}[Proof of Proposition \ref{profoubar2}]
For $k$ in $\m Z/d\m Z$, we compute replacing $\ell$ by $2\ell$ in Formula \eqref{eqnfoufor}, and  using Lemma \ref{lemfouthe} with $w=\la_0\overline{z}$ and with $k$ replaced by $4k$,
\begin{eqnarray*} 
\widehat{f}_{\overline{z}}(2k)
&=&\textstyle
\frac{1}{\sqrt{d}}\sum\limits_{\ell \in \m Z/d\m Z}
e^{-8i\pi k \ell/d}\;
\theta(\la_0\overline{z}\!+\!\ell /d,\tau_0)\\
&=&
\sqrt{d}\,e^{16i\pi k^2\tau_0}\,e^{8i\pi k \la_0\overline{z}}\;\theta(d\la_0\overline{z}+4dk\tau_0,d^2\tau_0)
\end{eqnarray*}
Applying the complex conjugation to this equality,
remembering that the theta functions are $1$-periodic,
and using the equalities
\begin{equation}
\label{eqnlazlaz}
4dk\overline{\tau}_0+dk
\;=\; dk/\la_0^2
\; =\; \overline{\la}_0k/\la_0,
\;\;\mbox{\rm one gets}
\end{equation}
\begin{eqnarray*} 
\overline{\widehat{f}}_{\overline{z}}(2k)
&=&
\sqrt{d}\,e^{-16i\pi k^2\overline{\tau}_0}\,e^{-8i\pi k \overline{\la}_0z}\;\theta(d\overline{\la}_0 
(z\!+\!\tfrac{k}{\la_0 d}),-d^2\overline{\tau}_0).
\end{eqnarray*}
We now apply Corollary \ref{cortrathe}
where $z$ is replaced by $z\!+\!\tfrac{k}{\la_0 d}$ and we get
\begin{eqnarray*} 
\overline{\widehat{f}}_{\overline{z}}(2k)
&=&\textstyle
\frac{\la_0}{\sqrt{d}}\,e^{-16i\pi k^2\overline{\tau}_0}\,e^{-8i\pi k \overline{\la}_0z}\;e^{4i\pi d^2(z\!+\!\tfrac{k}{\la_0 d})^2}\;
\theta(\la_0z\!+\!k/d,\tau_0).
\end{eqnarray*}
In the exponent the terms that are linear in $k$ are equal to
$$
8i\pi kz\,(d/\la_0-\overline{\la}_0)\;=\; 0,
$$
and the terms that are quadratic in $k$ are equal to 
$$
4i\pi k^2\,(1/\la_0^2-4\overline{\tau}_0)=4i\pi k^2 \in 2i\pi \m Z.
$$
Both of them disappear, and one gets
\begin{eqnarray*} 
\overline{\widehat{f}}_{\overline{z}}(2k)
&=&\textstyle
\frac{\la_0}{\sqrt{d}}\;e^{4i\pi d^2z^2}\;
\theta(\la_0z\!+\!k/d,\tau_0)\\
&=&
\textstyle
\frac{\la_0}{\sqrt{d}}\;e^{4i\pi d^2z^2}\;
f_{z}(2k),
\end{eqnarray*}
for all $k$ in $\m Z/d\m Z$. This proves \eqref{eqnfoubar2}.
\end{proof}


\begin{proof}[Proof of Theorem \ref{thmfoubar}]
It follows from Corollary \ref{corfoubar2},
that the $d$-critical value $\la_0:=\sqrt{a}+i\sqrt{b}$ is in $\mc B_d^o$.
Moreover, by choosing the parameter $r$ to be $0$ in 
Corollary \ref{corfoubar2}, the function $f$ satisfying 
\eqref{eqnbenoisto} with $\la=\la_0$ can be chosen to be symmetric
i.e. such that $f(k)=f(-k)$ for all $k$ in $\m Z/d\m Z$.

Therefore the complex conjugate function $\overline{f}$ which is $\overline{\la}_0$-critical 
also satisfies \eqref{eqnbenoisto} with $\la=\overline{\la}_0$. 
This prove that $\overline{\la}_0$ is also in $\mc B_d^o$.
\end{proof}

\section{Theta functions with characteristic}.
\label{secthecha}

In this chapter, we 
give a new proof of Proposition \ref{profzcri} 
by applying the same calculation as above to theta functions with characteristic. Most of the calculation relies only on the assumption 
that $a$ and $b$ are positive real numbers with $a\!+\!b=d$. 
The integrality condition is only needed 
at the end. 

\subsection{Reinterpreting the convolution equation}

The starting point is the  following Lemma \ref{lemcriff2}
which gives a reinterpretation of the conditions 
in \eqref{eqnbenoisto} thanks to the square $F=f^2$ of the function $f$.

\bl 
\label{lemcriff2}
Let $d$ be an odd integer,
$\la$, $\al$ be complex numbers, $f$  be a functions on $G=\m Z/d\m Z$
such that $\overline{\widehat{f}}(t)=\al\, f(t)$ and let $F=f^2$.
Then the following conditions are equivalent
\begin{equation}
\label{eqnfunequ}
f*f(2t)=\la \,f^2(t)
\hspace{1em}\Longleftrightarrow \hspace{1em}
\textstyle\overline{\widehat{F}}(2t)=\frac{\al^2\la}{\sqrt{d}}\, F(t)
\end{equation}
\el
In these equalities one must interpret $t$ as a variable
describing  $G$.

\begin{proof}[Proof of Lemma \ref{lemcriff2}]
We recall the formula which is  valid for any function $f$
\begin{equation}
\label{eqncriff2}
\overline{\widehat{F}}\;=\;
\tfrac{1}{\sqrt{d}}\;\overline{\widehat{f}}\!*\!\overline{\widehat{f}}.
\end{equation}
Since our function $f$ satisfies $\overline{\widehat{f}}=\al\, f$, this proves the claim \eqref{eqnfunequ}. 
\end{proof}

This lemma tells us that we need to compute the finite Fourier transform of the square $F_z$ of the theta functions $f_z$. This will be the aim of the next four sections.

\subsection{Square of theta functions} 

We need to introduce the theta 
functions with characteristic
\begin{eqnarray*}
\theta_{[0]}(z,\tau)\;\;=\;\;
\theta
\mbox{\scriptsize $\left[\!\!\!\begin{array}{c} 0 \\0\end{array}\!\!\!\right]$} (2z,2\tau)
&:=&\textstyle
\sum\limits_{m\,\rm even}e^{i\pi\frac{\tau}{2} m^2}e^{2i\pi mz}\\
\theta_{[1]}(z,\tau)\;=\;\theta
\mbox{\scriptsize $\left[\!\!\!\!
\begin{array}{c} 1/2\\0\end{array}\!\!\!\!\right]$} (2z,2\tau)
&:=&\textstyle
\sum\limits_{m\,\rm odd}e^{i\pi\frac{\tau}{2} m^2}e^{2i\pi mz}.
\end{eqnarray*}

Note that one has the equalities:
\begin{eqnarray}
\label{eqnth0th1}\textstyle
\theta_{[0]}(z,\tau)=\theta(2z,2\tau)
&{\rm and}&
\theta_{[1]}(z,\tau) =
e^{i\pi\tau/2}\;e^{2i\pi z}\;\theta_{[0]}(z\!+\!\tau/2,\tau).
\hspace*{2em}
\end{eqnarray}

We first recall that the square of the theta function is a linear combination of two theta functions with characteristic.

\bl 
\label{lemaddthe} For all $z$ in $\m C$, $\tau\in \m H$, one has
\begin{eqnarray}
\label{eqnaddthe}
\theta(z,\tau)^2&=&
\theta_{[0]}(0,\tau )\theta_{[0]}(z,\tau )+
\theta_{[1]}(0,\tau )\theta_{[1]}(z,\tau ).
\end{eqnarray}
\el 

\begin{proof} 
Just write the left-hand side as a double sum over $m$, $n$ in $\m Z$
and split this double sum according to the parity of $m\!-\! n$.
\end{proof}

The following calculation of the finite Fourier transform for these functions 
$\theta_{[0]}$ and $\theta_{[1]}$, is an analogue of Lemma \ref{lemfouthe}.

\bl 
\label{lemfout01}  For $w\in \m C$,  $\tau\in \m H$, $d>0$ odd integer
and $k\in \m Z/d\m Z$, one has
\begin{eqnarray*}
\label{eqnfout0}\textstyle
\sum\limits_{\ell \in \m Z/d\m Z}
e^{-4i\pi k \ell/d}\;\theta_{[0]}(w\!+\!\ell /d,\tau)
&=&
d\,e^{2i\pi k^2\tau}\,e^{4i\pi k w}\;\theta_{[0]}(dw+dk\tau,d^2\tau),\\
\textstyle
\sum\limits_{\ell \in \m Z/d\m Z}
e^{-4i\pi k \ell/d}\;\theta_{[1]}(w\!+\!\ell /d,\tau)
&=&
d\,e^{2i\pi k^2\tau}\,e^{4i\pi k w}\;\theta_{[1]}(dw+dk\tau,d^2\tau).
\end{eqnarray*}
\el

\begin{proof}
The proof is the same as for Lemma \ref{lemfouthe}, except that we have to restrict the sum to even integers $m$ for the first formula, and to odd integers $m$ for the second formula.
Here are the details that prove the first formula
\begin{eqnarray*} 
LHS&=&\textstyle\sum\limits_{\ell \in \m Z/d\m Z, m\,{\rm even}}
e^{i\pi\frac{\tau}{2} m^2}e^{2i\pi mw}e^{2i\pi \ell (m-2k)/d}\\
&=& \textstyle d \sum\limits_{n\,{\rm even}}
e^{i\pi \frac{\tau}{2} (2k+dn)^2}e^{2i\pi (2k+dn)w}
\; =\; RHS.
\end{eqnarray*}
And similarly for the second formula.
\end{proof}

\subsection{More transformation formulas }

We go on  with an analog of Corollary \ref{cortrathe}.

\begin{Cor}
\label{cortrat01}	Let  $d$ be an odd integer,  $a$,$b$ be positive real numbers with  $a\!+\!b\!=\!d$.
Let $\la_0:=\sqrt{a}\!+\!i\sqrt{b}$, and
$\displaystyle \tau_0:=\tfrac{\la_0^2-d^2}{4d^2}$.
Then, for all $z$ in $\m C$, one has 
\begin{eqnarray}
\label{eqntrat0}\textstyle
\theta_{[0]}(d\overline{\la}_0z,-d^2\overline{\tau}_0)
&=& \tfrac{\la_0}{d}\, e^{8i\pi d^2z^2}\;
\theta_{[0]}(\la_0z,\tau_0),\\
\label{eqntrat1}\textstyle
(-1)^{\frac{d^2-1}{8}} \;\theta_{[1]}(d\overline{\la}_0z,-d^2\overline{\tau}_0)
&=&\tfrac{\la_0}{d}\, e^{8i\pi d^2z^2}\;
\theta_{[1]}(\la_0z,\tau_0).
\end{eqnarray}
\end{Cor}

\begin{proof}[Proof of Corollary \ref{cortrat01}]
We first prove \eqref{eqntrat0}. We apply Lemma \ref{lemtrathe} with the matrix 
$$\si_1 =\mbox{\scriptsize $\left(\!
\begin{array}{cc} d^2&\frac{d^2-1}{2} \\2&1\end{array}\!
\right)$} \equiv {\bf 1}\; {\rm mod}\; 2,
$$ 
with $w=2\la_0 z$ and with $\tau_1:=2\tau_0$. 
This matrix $\si_1$ has been chosen so that one has the equality 
$\si_1\tau_1 =- d^2\overline{\tau}_1$. 
Remembering that $\la_0\overline{\la}_0=d$ and that 
$1/(2\tau_1\!+\!1)=\overline{\la}_0^2$, this gives the equality
$$
\theta(2d\overline{\la}_0z,-2d^2\overline{\tau}_0)
\; =\; \tfrac{\la_0}{d}\, e^{8i\pi d^2z^2}\;
\theta(2\la_0z,2\tau_0),
$$
which is nothing but  \eqref{eqntrat0}.

We now explain how to deduce \eqref{eqntrat1} from \eqref{eqntrat0}.
We compute the left-hand side of \eqref{eqntrat1}
up to multiplicative factors $M_1$, $M_2$ and $M_3$ 
whose values are given below.
\begin{eqnarray*}
\hspace*{5em}LHS&=&(-1)^{\frac{d^2-1}{8}} \;\theta_{[1]}(d\overline{\la}_0z,-d^2\overline{\tau}_0)\\
&=& M_1\;\theta_{[0]}(d\overline{\la}_0z-d^2\tfrac{\overline{\tau}_0}{2},-d^2\overline{\tau}_0)
\hspace{5em} \mbox{\rm by \eqref{eqnth0th1}}\\
&=& M_1\;\theta_{[0]}(d\overline{\la}_0(z+\tfrac{{\tau}_0}{2\la_0}),-d^2\overline{\tau}_0)
\hspace{4.8em} \mbox{\rm by \eqref{eqndtodto}}\\
&=& M_1M_2\;\theta_{[0]}({\la}_0(z+\tfrac{{\tau}_0}{2\la_0}),{\tau}_0)
\hspace{5.9em} \mbox{\rm by \eqref{eqntrat0}}\\
&=& M_1M_2M_3\;\theta_{[1]}({\la}_0z,{\tau}_0).
\hspace{7.7em} \mbox{\rm by \eqref{eqnth0th1}}
\end{eqnarray*}
Note that in the third line of this computation, we used the periodicity 
of the function $\theta_{[0]}$ and  the fact that the sum
\begin{equation} 
\label{eqndtodto}
d^2\tfrac{\overline{\tau}_0}{2}+\overline{\la}^2_0\tfrac{{\tau}_0}{2}
\; =\;
\tfrac{1-d^2}{8} 
\end{equation}
is an integer.
The factors are given by
\begin{eqnarray*}
M_1&=&e^{i\pi \frac{1-d^2}{8}} \;e^{-i\pi d^2\frac{\overline{\tau}_0}{2}}\;
e^{2i\pi d\overline{\la}_0 z},\\
M_2&=&\tfrac{\la_0}{d}\, e^{8i\pi d^2(z\!+\!\tfrac{{\tau}_0}{2\la_0})^2},\\
M_3&=& e^{-i\pi \frac{\tau_0}{2}}\, e^{-2i \pi \la_0 z}.
\end{eqnarray*}
In the exponent of the product $M_1M_2M_3$ the constant terms are equal to
$$
\tfrac{i\pi}{2}\;( \tfrac{1-d^2}{4}-d^2\overline{\tau}_0+4\overline{\la}_0^2\tau_0^2-\tau_0)
\; =\;
\tfrac{i\pi\tau_0}{2}\;( \overline{\la}_0^2+4\overline{\la}_0^2\tau_0-1)
\; =\; 0,
$$
and the terms that are linear in $k$ are equal to 
$$
2i\pi\la_0 z \; (\overline{\la}_0^2+4\overline{\la}_0^2\tau_0-1)
\; =\; 0.
$$
Both of them disappear, and one gets
$LHS\; =\;
\textstyle
\frac{\la_0}{d}\;e^{8i\pi d^2z^2}\;
\theta_{[1]}({\la}_0z,{\tau}_0).$
\end{proof}

\subsection{More Fourier transform of theta functions }

We begin by an analog of Proposition \ref{profoubar2}.

\begin{Prop}
\label{profoubar3}	Let $G\!=\!\m Z/d\m Z$ with $d$ odd,  $a$,$b$ be positive real numbers with 	$a\!+\!b\!=\!d$.
Let $\la_0:=\sqrt{a}\!+\!i\sqrt{b}$,
$\displaystyle \tau_0:=\tfrac{\la_0^2-d^2}{4d^2}$ and $z\in \m C$.
Then\\
$(i)$ the function $f_{0,z}:\ell\mapsto \theta_{[0]}(\la_0z\!+\!\frac{d+1}{2d}\ell,\tau_0)$  on $G$ satisfies,
for all $k$ in $G$,
\begin{eqnarray}
\label{eqnfoubar30}
\textstyle
\hspace{4em}
\overline{\widehat{f}}_{0,\overline{z}}(2k)
&=&
\tfrac{\la_0}{\sqrt{d}} \; e^{8i\pi d^2z^2}\, f_{0,z}(k).
\end{eqnarray}
$(ii)$ the function $f_{1,z}:\ell\mapsto \theta_{[1]}(\la_0z\!+\!\frac{d+1}{2d}\ell,\tau_0)$  on $G$ satisfies,
for all $k$ in $G$,
\begin{eqnarray}
\label{eqnfoubar31}\textstyle
(-1)^{\frac{d^2-1}{8}}\;	\overline{\widehat{f}}_{1,\overline{z}}(2k)
&=&
\tfrac{\la_0}{\sqrt{d}} \; e^{8i\pi d^2z^2}\, f_{1,z}(k).
\end{eqnarray}
\end{Prop}

\begin{proof}[Proof of Proposition \ref{profoubar3}]
The proof is the same as for Proposition \ref{profoubar2}.

We first prove \eqref{eqnfoubar30}.
For $k$ in $\m Z/d\m Z$, we compute  using Lemma \ref{lemfout01}  
\begin{eqnarray*} 
\widehat{f}_{0,\overline{z}}(4k)
&=&\textstyle
\frac{1}{\sqrt{d}}\sum\limits_{\ell \in \m Z/d\m Z}
e^{-16i\pi k \ell/d}\;
\theta_{[0]}(\la_0\overline{z}\!+\!\ell /d,\tau_0)\\
&=&
\sqrt{d}\,e^{32i\pi k^2\tau_0}\,e^{16i\pi k \la_0\overline{z}}\;\theta_{[0]}(d\la_0\overline{z}+4dk\tau_0,d^2\tau_0).
\end{eqnarray*}
And hence, applying first Equation \eqref{eqnlazlaz}  and then Corollary \ref{cortrat01}, one gets
\begin{eqnarray*} 
\overline{\widehat{f}}_{0,\overline{z}}(4k)
&=&
\sqrt{d}\,e^{-32i\pi k^2\overline{\tau}_0}\,e^{-16i\pi k \overline{\la}_0z}\;\theta_{[0]}(d\overline{\la}_0 
(z\!+\!\tfrac{k}{\la_0 d}),-d^2\overline{\tau}_0)\\
&=&\textstyle
\frac{\la_0}{\sqrt{d}}\,e^{-32i\pi k^2\overline{\tau}_0}\,e^{-16i\pi k \overline{\la}_0z}\;e^{8i\pi d^2(z\!+\!\tfrac{k}{\la_0 d})^2}\;
\theta_{[0]}(\la_0z\!+\!k/d,\tau_0).
\end{eqnarray*}
In the exponent the terms that are linear in $k$ are equal to
$$
16i\pi kz\,(d/\la_0-\overline{\la}_0)\;=\; 0,
$$
and the terms that are quadratic in $k$ are equal to 
$$
8i\pi k^2\,(1/\la_0^2-4\overline{\tau}_0)=8i\pi k^2 \in 2i\pi \m Z.
$$
Both of them disappear, and one gets
$\overline{\widehat{f}}_{0,\overline{z}}(4k)=
\frac{\la_0}{\sqrt{d}}\;e^{8i\pi d^2z^2}\;
f_{0,z}(2k).$

The proof of \eqref{eqnfoubar31} is similar, 
the sign
$(-1)^{\frac{d^2-1}{8}}$ sign coming from
\eqref{eqntrat1}.
\end{proof}

\subsection{Fourier transform of squares of theta functions}

We have not yet assumed that the positive real number $a$ is an integer
equal to $\tfrac{(d+1)^2}{4}$ modulo $4$. 
This condition will be crucial in the next Lemma \ref{lemfoubar4}.

\begin{Prop}
\label{profoubar4}
Let $G\!=\!\m Z/d\m Z$ with $d$ odd,  $a$,$b$ be positive integers with 
$a\!+\!b\!=\!d$ and $a\!\equiv\!\frac{(d+1)^2}{4}\;{\rm  mod}\; 4$.
Let $\la_0:=\sqrt{a}\!+\!i\sqrt{b}$,
$\displaystyle \tau_0:=\tfrac{\la_0^2-d^2}{4d^2}.$
The square function $F_{z}:\ell\mapsto \theta(\la_0z\!+\!\frac{d+1}{2d}\ell,\tau_0)^2$ on $G$ satisfies,
for all $k$ in $G$,
\begin{equation}\label{eqnfoubar4}
\textstyle
\overline{\widehat{F}}_{\overline{z}}(2k)=\frac{\la_0^3}{d\sqrt{d}} \,e^{8i\pi d^2z^2} \, F_{z}(k).
\end{equation}
\end{Prop}

The following lemma will be useful.

\begin{Lem}
\label{lemfoubar4}
Let $G\!=\!\m Z/d\m Z$ with $d$ odd,  $a$,$b$ be positive integers with 
$a\!+\!b\!=\!d$ and $a\!\equiv\!\frac{(d+1)^2}{4}\;{\rm  mod}\; 4$.
Let $\la_0:=\sqrt{a}\!+\!i\sqrt{b}$,
$\displaystyle \tau_0:=\tfrac{\la_0^2-d^2}{4d^2}.$ Then, one has
\begin{eqnarray}
\label{eqnth0bar}
\overline{\la}_0\,\theta_{[0]}(0,-\overline{\tau}_0)
&=&\la_0\,\theta_{[0]}(0,\tau_0)\\
\label{eqnth1bar}
(-1)^{\frac{d^2-1}{8}}\;\overline{\la}_0\,\theta_{[1]}(0,-\overline{\tau}_0)
&=&\la_0\,\theta_{[1]}(0,\tau_0).
\end{eqnarray}
\end{Lem}

Assertion \eqref{eqnth0bar} means that
$\la_0\,\theta_{[0]}(0,\tau_0)$ is real. 
Assertion \eqref{eqnth1bar} means that
$\la_0\,\theta_{[1]}(0,\tau_0)$ is real when 
$d\!\equiv\!\pm 1$ mod $8$ and imaginary otherwise.

\begin{proof}[Proof of Lemma \ref{lemfoubar4}]
The congruence assumption on $a$ tells us that 
\begin{equation}
\label{eqnd2td2t}
d^2\tau_0+d^2\ol{\tau}_0\;=\; \tfrac{a-b-d^2}{2}
\; \;\mbox{\rm belongs to}\;\; 
\tfrac{d^2-1}{4}\! +\!4\m Z.
\end{equation}
Combining this with the equality $\theta_{[0]}(0,\tau+2)=\theta_{[0]}(0,\tau)$ 
and using twice Formula \eqref{eqntrat0} one gets
$$
\tfrac{\overline{\la}_0}{d}\,\theta_{[0]}(0,-\overline{\tau}_0)
\;=\; 
\theta_{[0]}(0,d^2\tau_0)
\;=\; 
\theta_{[0]}(0,-d^2\overline{\tau}_0)
\;=\; 
\tfrac{\la_0}{d}\,\theta_{[0]}(0,\tau_0)
$$
which proves \eqref{eqnth0bar}.

Similarly with \eqref{eqnd2td2t} and the equality $\theta_{[1]}(0,\tau+2)=-\theta_{[1]}(0,\tau)$, 
one has the following computation in which we
use twice Formula  \eqref{eqntrat1},
$$
(-\!1)^{\frac{d^2\!-\!1}{8}}\tfrac{\overline{\la}_0}{d}\,\theta_{[1]}(0,-\overline{\tau}_0)
=
\theta_{[1]}(0,d^2\tau_0)
=
(-\!1)^{\frac{d^2\!-\!1}{8}}\theta_{[1]}(0,-d^2\overline{\tau}_0)
=
\tfrac{\la_0}{d}\theta_{[1]}(0,\tau_0)
$$
which proves \eqref{eqnth1bar}. 
\end{proof}

\begin{proof}[Proof of Proposition \ref{profoubar4}]
According to Lemma \ref{lemaddthe}, one has, for 
$k$ in $G$,
\begin{eqnarray}
F_z(k)&=& 	
\theta_{[0]}(0,\tau_0)\, f_{0,z}(k)+
\theta_{[1]}(0,\tau_0)\, f_{1,z}(k).
\end{eqnarray}
Therefore we compute the conjugate Fourier transform
\begin{eqnarray*}
\overline{\widehat{F}}_z(2k)&=& 	
\theta_{[0]}(0,-\overline{\tau}_0) \overline{\widehat{f}}_{0,z}(2k)+
\theta_{[1]}(0,-\overline{\tau}_0) \overline{\widehat{f}}_{1,z}(2k)\\
&=&
\tfrac{\la_0^2}{d}\left(
\theta_{[0]}(0,\tau_0) \overline{\widehat{f}}_{0,z}(2k)+
(-1)^{\frac{d^2-1}{8}}\;\theta_{[1]}(0,\tau_0) \overline{\widehat{f}}_{1,z}(2k)\right)\\
&=&
\tfrac{\la_0^3}{d\sqrt{d}}\,e^{8i\pi d^2z^2}\left(
\theta_{[0]}(0,\tau_0)f_{0,z}(k)+
\theta_{[1]}(0,\tau_0)f_{1,z}(k)\right)\\
&=&
\tfrac{\la_0^3}{d\sqrt{d}}\; e^{8i\pi d^2z^2}\,F_z(k).
\end{eqnarray*}
where we used 
Formulas 
\eqref{eqnth0bar} and \eqref{eqnth1bar}, and  Formulas
\eqref{eqnfoubar30} and \eqref{eqnfoubar31}.
\end{proof}

\subsection{Another proof of Proposition \ref{profzcri}}

\begin{proof}
By
Propositions \ref{profoubar2} and \ref{profoubar4}, 
we know that, for all $k$ in $G$,
\begin{eqnarray*}
\overline{\widehat{f}}_{\overline{z}}(k)
&=&
\al\, f_{z}(k)
\;\;{\rm with}\;\; \al := \tfrac{\la_0}{\sqrt{d}} \; e^{4i\pi d^2z^2},\\
\overline{\widehat{F}}_{\overline{z}}(2k)
&=&\be\, F_{z}(k)
\;\;{\rm with}\;\; 
\be := \tfrac{\la_0^3}{d\sqrt{d}} \; e^{8i\pi d^2z^2}.
\end{eqnarray*}
Therefore, by Formula \eqref{eqncriff2}, we get, for all $k$ in $G$,
\begin{eqnarray*}
f_z\!*\!f_z(2k)
&=&
\la\, f_z^2(k)
\;\;{\rm with}\;\; 
\la=\tfrac{\be\sqrt{d}}{\al^2}=\la_0.
\end{eqnarray*}
This says that the function $f_z$ is $\la_0$-critical.
\end{proof}

\section{Properties of the elements of $\mc B_d$}
\label{secproobd}

In this last chapter,  we explain the relationship between the set 
$\mc B^o_d$, the Gaussian functions, the Dirichlet characters, the Jacobi sums and the Weil numbers. We also update in section \ref{secliscri} the list of critical values for $d\leq 17$ that is given in \cite{CSAGI}. We will see in Section \ref{secnonwei} that this 
updated list for $d=17$ contains elements $\la\in \mc B^o_d$  that are not Weil numbers.

\subsection{Gaussian functions}
\label{secgaufun}

Let $d$ be an odd integer, $\ze:=e^{2i\pi/d}$ and $\eta:=-e^{i\pi/d}$,
so that $\eta^2=\ze$. Set $g_d$ to be the classical Gauss sum $g_d:=\sum_{k}\ze^{k^2}$. So that one has $g_d=\sqrt{d}$ when $d\equiv 1$ mod $4$,
and $g_d=i\sqrt{d}$ when $d\equiv 3$ mod $4$.
 
For $u\in (\m Z/d\m Z)^*$, $v\in \m Z/d\m Z$ we introduce 
the Gaussian function on $\m Z/d\m Z$
$$
f_{u,v}
(k)=\eta^{-u(k-v)^2}.
$$
The conjugate of its Fourier transform is given by
\begin{equation} 
\label{eqngaufun}
\overline{\widehat{f}}_{u,v} = \; (\!\tfrac{2u}{d}\!)\,\tfrac{\overline{g}_d}{\sqrt{d}}\, f_{u^{-1},uv}
\end{equation}

\bl 
\label{lemgaufun} 
$a)$ For $d$ odd, the classical Gauss sum $\la_0=g_d$ belongs to $\mc B_d^o$.\\
$b)$ When $d\equiv 1$ mod $4$, the opposite value $\la_0=-\sqrt{d}$ also belongs to $\mc B_d$.\\
$c)$ When $d\equiv 3$ mod $4$, the opposite value $\la_0=-i\sqrt{d}$ also belongs to $\mc B_d^o$ .
\el

\begin{proof}
For  $f=f_{u,v}$, one computes 
\begin{eqnarray*}
f\!*\! f(2k)&=&\la\, f^2(k)
\;\;{\rm with}\;\; 
	\la=(\!\tfrac{u}{d}\!)\, g_d.
\end{eqnarray*}
This says that the function $f$ is $\la_0$-critical
with $\la_0=(\!\tfrac{u}{d}\!)\, g_d$.

$a)$ When $u=1$ and $v=0$, one has $\la_0=g_d$, and Equation 
\eqref{eqngaufun},
tells us that the $\la_0$-critical function $f:k\mapsto \eta^{-k^2}$ is proportional to its conjugate Fourier transform $\overline{\widehat{f}}$.

$b)$ When  $u\in \m Z/d\m Z$ is not a square and $v=0$, 
one has $\la_0=-g_d$, and the function $f:k\mapsto \eta^{-uk^2}$ 
is $\la_0$-critical. 
Equation \eqref{eqngaufun},
tells us that the   conjugate Fourier transform $\overline{\widehat{f}}$
is proportional to $k\mapsto f(u^{-1}k)$.

$c)$ When $d\equiv 3$ mod $4$ and $u=-1$, one has $\la_0=-i\sqrt{d}$, and Equation 
\eqref{eqngaufun},
tells us that the $\la_0$-critical function $f:k\mapsto \eta^{k^2}$ is proportional to its conjugate Fourier transform $\overline{\widehat{f}}$.
\end{proof}

\br 
For $d=5$ the critical value $\la_0=-\sqrt{5}$
does not belong to $\mc B^o_d$. Indeed 
a direct calculation shows that, up to scalar,  
there are exactly $10$ 
$\la_0$-critical functions. These are the gaussian functions
$f_{u,v}$, with $u=\pm 2$ and $v\in \m Z/5\m Z$. 
By \eqref{eqngaufun}, none of them is proportional to its conjugate Fourier transform.
\er

\subsection{Dirichlet characters and Jacobi sums}
\label{secdircha}

We refer to \cite[\S 3.5]{IwaniecKowalski} and 
to \cite{Conrad} for the results of this section.

We recall that a Dirichlet character
$\chi:\m Z/d\m Z\ra \m C$ is a multiplicative character 
on the multiplicative group $(\m Z/d\m Z)^\star$ which is extended by $0$
on the set of zero divisors of $\m Z/d\m Z$. The group of Dirichlet characters
is isomorphic to $(\m Z/d\m Z)^\star$. Its order is given by the Euler function $\ph(d)$.
A Dirichlet character of $\m Z/d\m Z$ is said to be primitive if does not come from
a Dirichlet character of $\m Z/d'\m Z$ for a proper divisor $d'$ of $d$.
The number $N(d)$ of primitive Dirichlet characters 
can be easily computed by the formulas:\\
$N(d_1d_2)=N(d_1)N(d_2)$ when $d_1$ is coprime to $d_2$.\\
$N(p)=p\!-\!2$ and $N(p^r)=(p-1)^2p^{r-2}$ when $p$ is prime and $r\geq 2$.\\
In particular primitive Dirichlet characters exist if only if and only if 
$d\not\equiv 2$ mod $4$.

To Dirichlet characters $\chi$, $\chi_1$, $\chi_2$ 
one associates the Gauss sum 
$$
G(\chi)=\textstyle\sum_{k} e^{2i\pi k/d}\chi(k)
$$ 
and the Jacobi sum 
$$
J(\chi_1,\chi_2)=\textstyle\sum_{k}\chi_1(k)\chi_2(1-k).
$$

\bl
\label{lemdircha}
For all Dirichlet character $\chi$ whose square $\chi^2$ is primitive, 
the value $\la_\chi:=\chi(4)\, J(\chi,\chi)$ belongs to $\mc B_d^o$. 
Moreover the Dirichlet character $\chi$ is a 
$\la_\chi$-critical function on $\m Z/d\m Z$ which is proportionnal to 
$\overline{\widehat{\chi}}$. 
\el

\br 
The number $N_0(d)$ of Dirichlet characters whose square is primitive 
can be easily computed by the formulas:\\
$N_0(d_1d_2)=N_0(d_1)N_0(d_2)$ when $d_1$ is coprime to $d_2$.\\
$N_0(p)=p\!-\!3$ and $N_0(p^r)=(p-1)^2p^{r-2}$ when $p$ is an odd prime, 
and $r\geq 2$.\\
In particular, for $d$ odd, Dirichlet characters with primitive square exist if only if and only if 
$d\not\equiv \pm 3 $ mod $9$.
\er

\begin{proof} 
The Fourier transform of a primitive Dirichlet character $\chi$ 
is proportional to the conjugate Dirichlet character
$$
\widehat{\chi}(k)=\tfrac{1}{\sqrt{d}}\sum_k e^{-2i\pi k\ell/d}\chi(\ell)
=\tfrac{\chi(-1)}{\sqrt{d}}G(\chi)\;\overline{\chi}(k).
$$
Hence, since both $\chi$ and $\chi^2$ are primitive on $\m Z/d\m Z$, one has
\begin{eqnarray*}
\overline{\widehat{\chi}}(k)
&=&
\al\, \chi(k)
\;\;{\rm with}\;\; \al := \tfrac{1}{\sqrt{d}} \;G(\overline{\chi}),\\
\overline{\widehat{\chi^2}}(2k)
&=&\be\, \chi^2(k)
\;\;{\rm with}\;\; 
\be := \tfrac{\chi(4)}{\sqrt{d}} \; G(\overline{\chi}^2).
\end{eqnarray*}
Therefore, by Formula \eqref{eqncriff2}, we get, for all $k$ in $G$,
\begin{equation}
\label{eqnchichichi}
\chi\!*\!\chi(2k)
=
\la\, \chi^2(k)
\;\;\;{\rm with}\;\;\; 
\la=\tfrac{\be\sqrt{d}}{\al^2}=
\tfrac{d\chi(4)\,G(\overline{\chi}^2)}{G(\overline{\chi})^2}=
\tfrac{\chi(4)\,G(\chi)^2}{G(\chi^2)}.
\end{equation}
This says that the function $\chi$ is $\la$-critical
and that $\la$ belongs to $\mc B_d^o$.
Evaluating Formula \eqref{eqnchichichi} at $k=1$ also gives
the equality
$\la=\chi(4)\, J(\chi,\chi)$.
\end{proof}	

Here are the values $\la_\chi$ obtained when $d\leq 17$.
Note that two different Dirichlet characters $\chi$ may have the same $\la_\chi$.

\noindent 
$\star$ $d=5.$ $N_0(d)=2$ and $\la_\chi=1\!\pm\! 2i$.\\
$\star$ $d=7.$ $N_0(d)=4$ and $\la_\chi=\pm 2\!\pm\! i\sqrt{3}$.\\
$\star$ $d=9.$ $N_0(d)=4$ and $\la_\chi=3$.\\
$\star$ $d=11.$ $N_0(d)=8$ and 
$\la_\chi=\pm (1\!+\!\eps\sqrt{5})\pm i\sqrt{5\!-\!2\eps \sqrt{5}}$
with $\eps=\pm 1$.\\
$\star$ $d=13.$ $N_0(d)=10$ and $\la_\chi=\pm 3\!\pm\! 2i$
and $-1\!\pm\! 2i\sqrt{3}$.\\
$\star$ $d=17.$ $N_0(d)=14$ and $\la_\chi= 3\!\pm\! 2i\sqrt{2}$
and $-1\!\pm\! 4i$ and\\
\hspace*{11.5em}$\la_\chi=-1\!+\!2\eps\sqrt{2}\pm 2i\sqrt{2\!+\!\eps \sqrt{2}}$
with $\eps=\pm 1$.

\subsection{Weil numbers} 
\label{secweinum}

We focus on the arithmetic properties of critical values.
Let $d$ be an odd integer and $G=\m Z/d\m Z$. 
We recall that a $d$-Weil number is an algebraic integer 
all of whose Galois conjugates have absolute value $\sqrt{d}$.

All the elements of $\mc B_d$ we have constructed so far belong to 
a cyclotomic field.  This is not always the case. Indeed, one can check
that the element $\la=1\!+\!\sqrt{5}+i\sqrt{9\!-\!2\sqrt{5}}$ 
belongs to	$\mc B_{15}$.
We may ask, for any odd integer $d$,
\vs 

\centerline{
Is every  $\la\in \mc B_d$ a $d$-Weil number?} 
\vs 

We will give a very partial positive answer to this question in 
Lemma \ref{lemweinum}
but we will exhibit an example of  $\la$ in $\mc B^o_d$  
which is not a Weil numbers in  Section \ref{secnonwei}. 
In this example, one has $d=17$.

Let $G$ be a finite abelian group of odd order $d$.
For any symmetric non degenerate pairing ${\bf e}(.,.):G\times G\ra \m S^1$,
where $\m S^1:=\{z\in \m C\mid |z|=1\}$,
one can define the conjugate Fourier transform $f\mapsto \overline{\widehat{f}}$ on
$\mc F_G$ by, for all  $x\in G$,
$$
\textstyle
\overline{\widehat{f}}(x)= \frac{1}{\sqrt{d}}
\textstyle\sum\limits_{y\in G}\overline{f(y)}\; {\bf e}(x,y).
$$
Since the pairing is symmetric and non-degenerate, this map $f\mapsto \overline{\widehat{f}}$ is again  an antilinear involution of $\mc F_G$.
Moreover for all $\la$  with $|\la|=\sqrt{d}$ the conjugate Fourier transform preserves the variety $\mc F_\la$ of $\la$-critical functions.

Therefore it is also natural to introduce the following finite set of numbers
\begin{eqnarray*} 
	\label{eqnbenoistg} 
	\mc B_G:&=&\{ \mbox{$\la\in \m C \mid$
		there exists a symmetric non-degenerate pairing on $G$}\\ 
	&&\hspace{5em}\mbox{and a $\la$-critical function $f$ on $G$
		such that $\overline{\widehat{f}}=f$}\}.
\end{eqnarray*} 
By definition, when $G=\m Z/d\m Z$, we have $\mc B_d=\mc B_G$.
Lemma \ref{lemweinum}
below gives some very partial insight on the set 
$\mc B_G$.

\bl
\label{lemweinum}
Let $G$ be a finite abelian group of odd order $d$.\\ 
$(i)$ If  $\la\in \mc B_G$ belongs to a CM-field then $\la$ is a $d$-Weil number.\\
$(ii)$ Moreover if the function $f$ satisfying \eqref{eqnbenoisto} can also be chosen in a CM-field then
all the Galois conjugates of $\la$ belong to $\mc B_G$.
\el

By definition, a CM-field is a totally complex quadratic extension of a totally real number field.

\begin{proof} 
	$(i)$ We first recall that, by \cite[Prop. 2.1]{CSAGI}, all the critical numbers
	$\la\in \mc C_G$ are algebraic integers.
	On the one hand, for each choice of symmetric non-degenerate pairing ${\bf e}(.,.)$, 
	the conjugate Fourier transform $f\mapsto \overline{\widehat{f}}$ is an isometry of 
	$\mc F_G$ endowed with the $L^2$-norm. Therefore all the numbers $\la$ 
	in $\mc B_G$ have absolute value equal to $\sqrt{d}$. 
	On the other hand for $\la$ in a CM-number field, all the Galois conjugates of $\la$ have the same absolute value.
	This proves that $\la$ is a $d$-Weil number.
	
	$(ii)$ Let $\la'$ be a Galois conjugate of $\la$.
	Let $K\subset \m C$ be a CM number field containing not only $\la$ and the $d^{\rm th}$-roots of unity but also the coefficients of a function $f$
	satisfying \eqref{eqnbenoisto}. We can assume that the extension $K/\m Q$ is Galois.
	Let $\si\in {\rm Gal}(K/\m Q)$ be an automorphism of $K$ such  that $\la'=\si(\la)$.

	Since $K$ is a CM number field, this element $\si$
	commutes with the complex conjugation.
	There exists an integer $m$ coprime to $d$ such that, for all $d^{\rm th}$-root of unity $\zeta$, one has $\si(\zeta)=\zeta^m$. 
	We then introduce the symmetric non-degenerate pairing on $G$
	$$
	{\bf e}_m(x,y)={\bf e}(x,y)^m={\bf e}(mx,y)={\bf e}(x,my)=\si({\bf e}(x,y))
	$$
	By assumption there exists a non-zero $\la$-critical function $f$ on $G$ and a constant $\al$ such that, 
	for all  $x\in G$,
	$$
	\textstyle
	\sum\limits_{y\in G}\overline{f(y)}\;{\bf e}(x,y) =\alpha f(x)
	$$
	Since $\si$ commutes with the complex conjugation, applying $\si$ to this equality, one gets, with the function $g:=f^\si$ which is $\la^\si$-critical,
	the equality
	$$
	\textstyle
	\sum\limits_{y\in G}\overline{g(y)}\;{\bf e_m}(x,y) =\alpha^\si g(x)
	$$
	This proves that $\la^\si$ belongs to $\mc B_G$.
\end{proof}	


\subsection{List of critical values for small $d$} 
\label{secliscri}

\noindent 
The following lists of $d$-critical values rely on numerical experiments using the Buchberger's
algorithm for computing Groebner basis (see \cite[Chap. 2]{CLO92}). Implementing the algorithm modulo many prime numbers allowed us to improve the lists given in \cite{CSAGI}.
\vs

For $3\leq d\leq 11$ here are the complete lists of $d$-critical values. For $d= 13$, the  list below of $d$-critical values is also probably complete
\vs 

\noindent 
For \fbox{$d=3$} \; $\la=1$, $3$, $\pm i\sqrt{3}$.

\noindent 
For  \fbox{$d=5$} \; $\la=1$, $5$, $\pm \sqrt{5}$, $1\pm 2i$.

\noindent 
For  \fbox{$d=7$} \; $\la=1$, $7$, $\pm i \sqrt{7}$,  $\pm 2\pm i\sqrt{3}$.

\noindent 
For \fbox{$d=9$} \; $\la=1$, $9$, $\pm i\sqrt{3}$, $\pm 3i\sqrt{3}$,\\
\hspace*{1em} $\la=3$, 
$\pm \sqrt{5}\pm 2i$, 
$\pm 1\pm 2i\sqrt{2}$.

\noindent 
For \fbox{$d=11$} \; $\la =1$ , $11$, $4\pm \sqrt{5}$, \\
\hspace*{1em} $\la=\pm i\sqrt{11}$, 
$2\pm i\sqrt{7}$, 
$\pm 2\sqrt{2}\pm i\sqrt{3}$,\\
\hspace*{1em} $\la=\pm (1\!+\!\eps\sqrt{5})\pm i\sqrt{5\!-\!2\eps \sqrt{5}}$ with $\eps=\pm 1$.

\noindent 
For \fbox{$d=13$} \; $\la =1$ , $13$, $5\pm 2\sqrt{3}$,\\
\hspace*{1em}  $\la=\pm \sqrt{13}$,
$\pm 3\pm2i$,
$\pm \sqrt{5}\pm 2i\sqrt{2}$, 
$\pm 1\pm 2i\sqrt{3}$.
\vs 

\noindent For \fbox{$d=15$}, here is the complete list of  $d$-critical values associated with a 
symmetric critical function.\\ 
$\star$ 
$\la=$ product of a $3$-critical and a $5$-critical value, $\la= 6\pm \sqrt{21}$,\\
$\star$ $\la=-3$, $-5$, 
$\pm(\sqrt{3}\pm i\sqrt{2})(\sqrt{2}\pm i)$,
$\pm 2\pm i\sqrt{11}$, $\pm 2\sqrt{2}\pm i\sqrt{7}$,\\
$\star$   $\la=1\!+\!\eps\sqrt{5}\pm i\sqrt{9\!-\!2\eps \sqrt{5}}$
or $\pm\sqrt{10+2\eps\sqrt{5}}\pm i\sqrt{5-2\eps\sqrt{5}}$ 
with $\eps=\pm 1$,\\
$\star$  $\la$ is a root of 
the equation $\la^2-4 b\la+15=0$ 
where either 
\begin{eqnarray*}
b^3-2b^2-2\;=\; 0\hspace{2em} {\rm  or}
\hspace{2em}
b^3-b^2-b-1&=&0\;\;\; {\rm  or}\\
b^{10}\!-\!b^9\!-\!10b^8\!+\!10b^7
\!+\!34b^6\!-\!38b^5\!-\!43b^4
\!+\!65b^3\!+\!8b^2\!-\!40b\!+\!16&=&0.
\end{eqnarray*}

\noindent For \fbox{$d=17$}, here is 
a list of  $d$-critical values 
that contains all those associated with an antisymmetric critical function.\\
$\star$ 
$\la =1$ , $17$, $7\pm 4\sqrt{2}$,\\
$\star$ $\la= \pm \sqrt{17}$,
$\pm \sqrt{13}\pm 2i$,
$3\pm 2i\sqrt{2}$,  
$\pm \sqrt{5}\pm 2i\sqrt{3}$, 
$\pm 1\pm 4i$,\\
$\star$ $\la=\pm(1+2\eps\sqrt{2})\pm 2i\sqrt{2-\eps\sqrt{2}}$ with $\eps=\pm 1$,\\
$\star$  $\la$ is a root of 
the equation 
\begin{equation} 
	\label{eqnnonwei} 
	\la^2-2 c\la+17=0\;\;\; {\rm where}
\end{equation}
\begin{equation*}
c^{10}\!-\!6c^9\!-\!15c^8\!+\!
136c^7\!-\!62c^6\!-\!628c^5\!+\!
586c^4\!+\!232c^3\!+\!
733c^2\!-\!246c\!+\!293=0.
\end{equation*}

\subsection{A critical value which is not a Weil number} 
\label{secnonwei}

\bexa
\label{exanonwei} For $d=17$, there exists a critical value 
$\la\in \mc B^o_d$ which is not a Weil number.
\eexa

\begin{proof} We choose $\la$ to be
one of the $20$  critical values $\la$ given by \eqref{eqnnonwei}.
These values $\la$
are not Weil numbers. Indeed, among these $20$ Galois conjugates, only $8$ have absolute value equal to
$\sqrt{17}$. Those for which $c$ is real and  belongs 
to the interval $[-\sqrt{17},\sqrt{17}]$.

Using again Buchberger's algorithm for computing Groebner basis,
one can list, for each such $\la$,  all the antisymmetric  $\la$-critical functions $f$ with $f(1)=1$. There are $8$ of them. One can then check that, for a  well-chosen $\la$ approximately equal to $ 3.942+ 1.209 \, i$, one of these $\la$-critical  functions $f$ is proportional to  its conjugate Fourier transform $ \overline{\widehat{f}}$. Such a critical value $\la$
belongs to $\mc B^o_d$.
\end{proof}	
\newpage

{\small
	\bibliography{theta3}
}
\vs 

{\small
	\noindent
	Y. \textsc{Benoist}: CNRS, 
	Universit\'e Paris-Saclay,\hfill
	\texttt{yves.benoist@u-psud.fr}}

\end{document}